\documentclass[]{amsart}

\usepackage[utf8]{inputenc}
\usepackage{amsthm,amssymb,amsmath,mathtools,mathpazo,xargs,graphicx,subcaption}
\usepackage[linecolor=orange,backgroundcolor=orange!25,bordercolor=orange]{todonotes}
\usepackage[colorlinks]{hyperref}

\title{Mini-course: Property of rapid decay}

\author{{\L}ukasz Garncarek}
\address{Institute of Mathematics\\ Polish Academy of Sciences\\ ul. Śniadeckich 8\\ 00-656 Warszawa (Poland)}
\email{lukgar@impan.pl} 

\thanks{This work was partially supported by the European Research
  Council (ERC) grant no. 259527 of G. Arzhantseva, and by the Polish
  National Science Center grant 2012/06/A/ST1/00259. The stay in
  Vienna was partially supported by the Erwin Schr\"odinger Institute.
}

\newtheorem{theorem}{Theorem}[section]

\newtheorem*{conjecture}{Conjecture}
\newtheorem{lemma}[theorem]{Lemma}
\newtheorem{proposition}[theorem]{Proposition}
\newtheorem{corollary}[theorem]{Corollary}
\theoremstyle{definition}
\newtheorem{definition}{Definition}
\newtheorem*{question}{Question}
\theoremstyle{remark}

\numberwithin{equation}{section}

\newcommand{\RR}{\mathbb{R}}
\newcommand{\ZZ}{\mathbb{Z}}
\newcommand{\NN}{\mathbb{N}}
\newcommand{\CC}{\mathbb{C}}

\DeclareMathOperator{\supp}{supp}
\DeclareMathOperator{\Stab}{Stab}

\newcommand{\norm}[1]{\left\lVert#1\right\rVert}
\newcommand{\abs}[1]{\left\lvert#1\right\rvert}
\newcommand{\set}[1]{\left\{#1\right\}}

\newcommand{\rc}{\widetilde{rc}}

\begin{document}

\maketitle

\begin{abstract}
  These are the notes for the mini-course on the property of rapid
  decay given at the Erwin Schr\"odinger Institute in Vienna between
  March 14th and 18th, 2016, as a part of the Measured Group Theory
  program. They include parts which were omitted during the lectures,
  and omit some parts which were included in the lectures (and now
  I think that they shouldn't be). We explore the definitions of
  property RD and its permanence properties, discuss the centroid
  property and relative centroid property defined by Mark Sapir, and
  give proofs that certain classes of groups have RD.
\end{abstract}

\section{Introduction}
\label{sec:introduction}

Let \(G\) be a group. For simplicity, we will restrict ourselves to finitely generated groups. With a suitable \emph{length function} \(\ell\) on \(G\), one may associate the \emph{space of rapidly decreasing functions} \(H_{\ell}^{\infty}(G)\). In short, property of rapid decay, abbreviated to \emph{property RD}, says that this space canonically embeds into \(C^{*}_{r}(G)\), the reduced \(C^{*}\)-algebra of \(G\). If, for example, \(G=\ZZ\), then \(C^{*}_{r}(G)\) consists of the Fourier transforms of continuous functions on the circle, and \(H_{\ell}^{\infty}(G)\) with \(\ell(x)=\abs{x}\) turns out to be the space of Fourier transforms of smooth functions on the circle. Thus, \(\ZZ\) has property RD, which can be thought of as a non-commutative analogue of the fact that smooth functions are continuous.

An easy consequence of property RD is that \(H_{\ell}^{\infty}(G)\) is a convolution algebra. We can therefore speak of its \(K\)-theory. It turns out that the embedding of \(H^{\infty}_{\ell}(G)\) into \(C^{*}_{r}(G)\) induces isomorphisms in \(K\)-theory \cite{Jolissaint1989}. This is relevant for the Baum-Connes conjecture, stating that a certain \emph{assembly map} \(\mu\colon K_{i}^{G}(\underline{E}G)\to K_{i}(C^{*}_{r}(G))\) is an isomorphism. Namely, for ``good'' groups with property RD, it is possible to construct an isomorphism \(K_{i}^{G}(\underline{E}G)\to K_{i}(H^{\infty}_{\ell}(G))\), whose composition with the induced isomorphism \(K_{i}(H^{\infty}_{\ell}(G)) \to K_{i}(C^{*}_{r}(G))\) is the assembly map \cite{Lafforgue2002}. 

These results were the source of initial motivation for introducing and studying property RD.

\section{Various formulations of property RD}
\label{sec:vari-form-prop}

In this section we will explore the equivalence of four different definitions of property RD. The first one is based on an estimate for the operator norm of elements of the group algebra acting on the space of square-integrable functions on the group.

The second and third definitions are formulated in terms of existence of continuous embeddings of certain Sobolev-type spaces into the group \(C^{*}\)-algebra.

Finally, the fourth definition is based on a decay estimate for matrix coefficients of the regular representation.

\subsection{Estimates on the operator norm of convolution operators}
\label{sec:estim-oper-norm}

Let \(G\) be a finitely generated group. A \emph{length function} on \(G\) is a map \(\ell\colon G\to [0,\infty)\) satisfying the following conditions:
\begin{enumerate}
\item \(\ell(1)=0\),
\item \(\ell(x^{-1})=\ell(x)\) for all \(x\in G\),
\item \(\ell(xy)\leq\ell(x)+\ell(y)\) for all \(x,y\in G\).
\end{enumerate}
A length function is \emph{proper} if the \(\ell\)-balls \(\set{x\in G : \ell(x)\leq C}\) are finite for every \(C>0\).

A standard example of length function is the word-length. If \(S\) is a symmetric generating set of \(G\), the word-length \(\ell_{S}\) is defined as
\begin{equation}
  \ell_{S}(x) = \min \{ n : x\in S^{n}\}.
\end{equation}
Another natural example is obtained from an isometric action of \(G\) on a metric space \((X,d)\). Any choice of a basepoint \(p\in X\) gives rise to a length function
\begin{equation}
  \ell_{p}(x) = d(p, xp).
\end{equation}
In fact, all length functions can be obtained from this construction.

For notational purposes we extend all length functions to the group algebra \(\CC[G]\) by putting
\begin{equation}
  \ell(f) = \max\{\ell(x) : f(x)\ne 0\}.
\end{equation}
In other words, this can be thought of as the radius of the support of \(f\).

If \(\ell_{1}\) and \(\ell_{2}\) are two length functions on \(G\), we say that \(\ell_{1}\) \emph{dominates} \(\ell_{2}\) if there exist constants \(C,s>0\) such that for all \(x\in G\)
\begin{equation}
  \ell_{2}(x)\leq C(1+\ell_{1}(x))^{s}.
\end{equation}

Observe, that in a finitely generated group the word-lengths dominate all other lengths, as for \(x\in G\) represented by a reduced word \(s_{1}\cdots s_{n}\) in generators from \(S\) we have
\begin{equation}
  \ell(x)=\ell(s_{1}\cdots s_{n}) \leq \sum_{i}\ell(s_{i}) \leq n \max_{s\in S}\ell(s) = \ell_{S}(x) \max_{s\in S}\ell(s).
\end{equation}

Now, let us proceed to the first definition of property RD. For
\(\CC[G]\) by \(\norm{f}_{2}\) we denote the \(\ell^{2}\)-norm of
\(f\), and by \(\norm{f}_{op}\) the operator norm of \(f\) acting on
\(\ell^{2}(G)\) by convolution (equivalently, by the left regular
representation \(\lambda\))
\begin{equation}
  f*\xi(x)=[\lambda(f)\xi](x)=\sum_{y\in G} f(y)\xi(y^{-1}x).
\end{equation}

\begin{definition} \label{def:norm-estimate}
  The group \(G\) satisfies \emph{property RD} with respect to a length function \(\ell\) is there exists a polynomial \(P(x)\) with positive coefficients, such that for all \(f\in\CC[G]\) the operator norm of \(f\) can be estimated by
  \begin{equation}
    \norm{f}_{op}\leq P(\ell(f))\norm{f}_{2}.
  \end{equation}
\end{definition}

We immediately see that if \(\ell_{2}\) dominates \(\ell_{1}\), then property RD with respect to \(\ell_{1}\) implies property RD with respect to \(\ell_{2}\):
\begin{equation}
  \norm{f}_{op} \leq P(\ell_{1}(f))\norm{f}_{2} \leq P(C(1+\ell_{2}(f))^{s})\norm{f}_{2} \leq  Q(\ell_{2}(f))\norm{f}_{2}.
\end{equation}
In particular, if \(G\) has property RD with respect to some length \(\ell\), then it has property RD with respect to the word-lengths. In this case we say that \(G\) has property RD, without specifying the length function.

\subsection{First examples}
\label{sec:first-examples}

Before continuing our exploration of the equivalent definitions of property RD, we will look at some examples.

\begin{proposition}\label{thm:poly-growth-rd}
  If \(G\) has polynomial growth with respect to a proper length function \(\ell\), then it has property RD with respect to \(\ell\).
\end{proposition}

\begin{proof}
  Let \(\gamma(n)=\abs{\set{x\in G : \ell(x) \leq n}}\) be the growth function of \(G\) with respect to \(\ell\). We have \(\gamma(n)\leq P(n)\) for some polynomial \(P\). Now, for \(f\in \CC[G]\) and \(\xi\in l^{2}(G)\), we get
  \begin{equation}
    \begin{split}
      \norm{f*\xi}_{2}^{2} & = \sum_{x\in G}\abs{\sum_{y\in
          G}f(y)\xi(y^{-1}x)}^{2} \leq \sum_{x\in G}\left(
        \norm{f}_{2}^{2} \sum_{y\in \supp{f}}
        \abs{\xi(y^{-1}x)}^{2}\right) = \\
      & = \norm{f}_{2}^{2}\sum_{y\in\supp{f}}\norm{\xi}_{2}^{2} \leq \gamma(\ell(f))\norm{f}_{2}^{2}\norm{\xi}_{2}^{2}\leq P(\ell(f)) \norm{f}_{2}^{2}\norm{\xi}_{2}^{2}. \qedhere
    \end{split} 
\end{equation}
\end{proof}

\begin{proposition} \label{prop:rd-amenable}
  Suppose \(G\) is amenable. Then it has property RD with respect to a proper length function \(\ell\) if and only if it has polynomial growth with respect to \(\ell\).
\end{proposition}

\begin{proof}
  The implication from right to left is true for any \(G\) by Proposition~\ref{thm:poly-growth-rd}. Now, assume that \(G\) is amenable, and has property RD with respect to \(\ell\). Denote by \(\gamma(n)\) the corresponding growth function. Define
  \begin{equation}
    f_{n}(x)=
    \begin{cases}
      1,& \text{if \(\ell(x)\leq n\)}\\
      0,&\text{otherwise}.
    \end{cases}
  \end{equation}
  By properness of \(\ell\), the function \(f_n\) is finitely supported. Clearly, \(\gamma(n)=\norm{f_{n}}_{1}\). Now we will apply Kesten's characterization of amenability, which states that in an amenable group for \(f\in \RR_{+}[G]\) one has \(\norm{f}_{op}=\norm{f}_{1}\). This way we obtain
  \begin{equation}
    \gamma(n)=\norm{f_{n}}_{1}=\norm{f_{n}}_{op} \leq P(n)\norm{f_{n}}_{2} = P(n) \gamma(n)^{1/2}.
  \end{equation}
Hence,
\begin{equation}
  \gamma(n) \leq P(n)^{2},
\end{equation}
and \(G\) has polynomial growth with respect to \(\ell\).
\end{proof}

Summarizing, if we take the word-length, we see that groups of polynomial growth always have property RD, groups of intermediate growth never do, and groups of exponential growth may have property RD, only if they are non-amena\-ble.

\subsection{Sobolev-type spaces}
\label{sec:sobolev-type-spaces}

Now, let us return to our definitions. The second one is purely technical, and will serve us in proving equivalence of the following ones.

For a length function \(\ell\), and \(s\in\RR\) we define the (pre-Hilbert) norm \(\norm {\cdot}_{\ell,s}\) on \(\CC[G]\) by
\begin{equation}
  \norm{f}_{\ell,s}^{2} = \sum_{x\in G} \abs{f(x)}^{2}(1+\ell(x))^{s}.
\end{equation}
The corresponding completion of \(\CC[G]\) will be denoted by \(H_{\ell}^{s}(G)\).

\begin{definition} \label{def:h-s}
  The group \(G\) has property RD with respect to a length function \(\ell\) if for some \(s>0\) the inclusion \(\CC[G]\subseteq C^{*}_{r}(G)\) of the group algebra into the reduced group \(C^{*}\)-algebra of \(G\) extends to a continuous operator \(H_{\ell}^{s}(G)\to C^{*}_{r}(G)\).
\end{definition}

\begin{proposition}
  Definitions \ref{def:norm-estimate} and \ref{def:h-s}are equivalent.
\end{proposition}

\begin{proof}
  If \(G\) satisfies Definition 2, then there exist constants \(C,s>0\) such that for \(f\in \CC[G]\) we have
  \begin{equation}
    \norm{f}_{op}^{2}\leq C \norm{f}_{\ell,s}^{2} = C\sum_{x\in G} \abs{f(x)}^{2}(1+\ell(x))^{s}\leq C(1+\ell(f))^{s} \norm{f}_{2}^{2}, 
  \end{equation}
  yielding property RD in the sense of Definition 1.

  Now, assume that all \(f\in\CC[G]\) satisfy the estimate
  \begin{equation}
    \norm{f}_{op}\leq P(\ell(f))\norm{f}_{2}
  \end{equation}
  for some polynomial \(P\). Then, there exist \(C,s>0\) such that for \(t\geq 0\) we have \(P(t)\leq C(1+t)^{s}\). If we denote \(A_{n}=\set{x\in G : \ell(x)\in [n,n+1)}\), and \(f_{n} = f|_{A_{n}}\), we may estimate
  \begin{equation}
    \begin{split}
      \norm{f}_{op} & \leq \sum_{n} \norm{f_{n}}_{op} \leq \sum_{n}
      P(n)\norm{f_{n}}_{2} = \sum_{n} (1+n)P(n)\norm{f_{n}}_{2}
      (1+n)^{-1} \leq \\
      & \leq \left( \sum_{n} (1+n)^{-2}\right)^{1/2} \left( \sum_{n} (1+n)^{2}P(n)^{2}\norm{f_{n}}_{2}^{2}\right)^{1/2} \leq\\
      & \leq \frac{\pi}{\sqrt{6}} \left( \sum_{n} C(1+n)^{s}\norm{f_{n}}_{2}^{2} \right)^{1/2} 1\leq C' \norm{f}_{\ell,s},
    \end{split}
  \end{equation}
  where \(C\) and \(s\) satisfy \((1+t)^{2}P(t)^{2}\leq C(1+t)^{s}\).
\end{proof}

Now, observe that the spaces \(H_{\ell}^{s}(G)\) together with the canonical inclusions \(j_{st}\colon H_{\ell}^{s}(G)\to H_{\ell}^{t}(G)\), where \(s>t\), form an inverse system of Banach spaces and contractive maps. We may thus define the \emph{space of rapidly decreasing functions} on \(G\) with respect to \(\ell\) as the inverse limit
\begin{equation}
  H^{\infty}_{\ell}(G)=\varprojlim_{s} H^{s}_{\ell}(G).
\end{equation}
In our situation, this is just the intersection of the spaces \(H^{s}_{\ell}(G)\), equipped with the minimal topology for which the inclusions \(H^{\infty}_{\ell}(G)\subseteq H^{s}_{\ell}(G)\) are continuous.

\begin{definition} \label{def:h-infty}
  The group \(G\) has property RD with respect to a length function \(\ell\) if the inclusion \(\CC[G]\subseteq C^{*}_{r}(G)\) extends to a continuous operator \(H^{\infty}_{\ell}(G)\to C^{*}_{r}(G)\).
\end{definition}

The proof of equivalence of the former definitions with Definition \ref{def:h-infty} will be based on the following general lemma about inverse limits of sequences of Banach spaces.

\begin{lemma}\label{thm:inv-lim-operator}
  Suppose that \((X_{s})_{s\in\RR}\) is an inverse system of Banach spaces with injective bonding maps \(J_{st}\in {\mathcal {L}}(X_{s}, X_{t})\) for \(s>t\), and let
  \begin{equation}
    X=\varprojlim X_{s}.
  \end{equation}
Then for any Banach space \(Y\), and any bounded operator \(T\colon X\to Y\) there exists \(s\in\RR\), and a bounded operator \(T'\colon \overline{J_{s}(X)}\to Y\) such that \(T=T'J_{s}\), where \(J_{s}\colon X\to X_{s}\) is the natural projection of the inverse limit.
\end{lemma}

\begin{proof}
  The inverse limit will not change if we replace \((X_{s})\) by the cofinal  system \((X_{n})_{n\in\NN}\), hence the topology of \(X\) is compatible with the translation-invariant metric
  \begin{equation}
    d(x,y)=\sum_{n} \frac{1}{2^{n}}\min\{1, \norm{J_{n}(x-y)}_{X_{n}}\}.
  \end{equation}
Since \(T\) is continuous, there exists \(r>0\) such that the ball \(B_{X}(r)\) of radius \(r\) around \(0\) in \(X\) is mapped by \(T\) into the unit ball of \(Y\). Now, consider the metric
\begin{equation}
  d_{N}(x,y)=\sum_{n\leq N} \frac{1}{2^{n}}\min\{1, \norm{J_{n}(x-y)}_{X_{n}}\}
\end{equation}
on \(X\) (incompatible with the topology), and choose \(N\) sufficiently large, so that
\begin{equation}
  \abs{d(x,y)-d_{N}(x,y)} = \sum_{n>N} \frac{1}{2^{n}}\min\{1, \norm{J_{n}(x-y)}_{X_{n}}\} \leq  \sum_{n>N} \frac{1}{2^{n}} < \frac{r}{2}.
\end{equation}
Thanks to this estimate, the \(d_{N}\)-ball \(B_{(X,d_{N})}(r/2)\) is
contained in \(B_{(X,d)}(r)\), so \(T\) is continuous with respect to
the metric \(d_{N}\) on \(X\). But
\begin{equation}
  \begin{split}
    \min\{1, \norm{J_{N}(x-y)}_{X_{N}}\} & \leq \sum_{n \leq N}
    \frac{1}{2^{n}} \min\{1, \norm{J_{n}(x-y)}_{X_{n}}\} \leq \\
    & \leq
    \sum_{n\leq N} \frac{1}{2^{n}} \min\{1,
    \norm{J_{Nn}}\norm{J_{N}(x-y)}_{X_{N}}\} \leq \\
    & \leq \left(\sum_{n\leq N} \frac{1+\norm{J_{Nn}}}{2^{n}}\right) \min\{1,
    \norm{J_{N}(x-y)}_{X_{N}}\},
  \end{split}
\end{equation}
so \(d_{N}\) gives rise to the same topology on \(X\) as the pullback
of the norm on \(X_{N}\) through \(J_{N}\). With this norm however,
\(J_{N}\) is an isometric embedding of \(X\) into the Banach space
\(\overline{J_{N}(X)}\subseteq X_{N}\), and thus \(T\) can be extended
to a continuous operator \(T'\colon \overline{J_{N}(X)} \to Y\), as required.
\end{proof}

\begin{proposition}
Definitions \ref{def:h-s} and \ref{def:h-infty} are equivalent.  
\end{proposition}
\begin{proof}
  If \(G\) satisfies Definition 2, then for some \(s\) we have a bounded operator \(H_{\ell}^{s}(G)\to C^{*}_{r}(G)\) extending the inclusion of the group algebra. But the inclusion \(\CC[G]\subseteq H_{\ell}^{s}(G)\) factorizes as \(\CC[G]\subseteq H_{\ell}^{\infty}(G)\subseteq H_{\ell}^{s}(G)\), thus giving a continuous operator \(H_{\ell}^{\infty}(G)\to C^{*}_{r}(G)\).

On the other hand, if \(G\) satisfies Definition 3, Definition 2 is a consequence of Lemma~\ref{thm:inv-lim-operator}.
\end{proof}

In the Introduction we remarked that for groups with property RD the space of rapidly decreasing functions is a convolution algebra. We will now prove this.

\begin{proposition}
  If \(G\) satisfies property RD with respect to \(\ell\), then \(H^{\infty}_{\ell}(G)\) is closed under convolution.
\end{proposition}

\begin{proof}
  Let \(f,g\in H_{\ell}^{\infty}(G)\). For \(s>0\) we have
  \begin{equation}
    \begin{split}
      \norm{f*g}^{2}_{\ell,2s} & = \sum_{x\in G}\abs{\sum_{y\in G} f(y)g(y^{-1}x)(1+\ell(x))^{s}}^{2} \leq \\
    & \leq \sum_{x\in G}\abs{\sum_{y\in G} \abs{f(y)}(1+\ell(y))^{2}\abs{g(y^{-1}x)}(1+\ell(y^{-1}x))^{s}}^{2} = \\
    & = \norm{f_{s}*g_{s}}_{2}^{2},
    \end{split}
  \end{equation}
  where \(f_{s}(x)=\abs{f(x)}(1+\ell(x))^{2}\), and similarly with \(g_{s}\). Since \(f\) and \(g\) are rapidly decreasing, so is \(f_{s}\), and
  \begin{equation}
    \norm{f_{s}* g_{s}}_{2} \leq \norm{f_{s}}_{op} \norm{g}_{\ell,2s} < \infty.
  \end{equation}
  This means that \(f*g\in H^{s}_{\ell}(G)\) for all \(s\), so it is rapidly decreasing.
\end{proof}

\subsection{Some remarks}
\label{sec:some-remarks}

During the mini-course a question was raised, whether the embeddings
\(H^{s}_{\ell}(G)\to H^{t}_{\ell}(G)\) for \(s>t\) are compact
operators. The following proposition answers this question.

\begin{proposition}
  The canonical embedding \(H^{s}_{\ell}(G)\to H^{t}_{\ell}(G)\) is
  compact if and only if the length \(\ell\) is proper.
\end{proposition}

\begin{proof}
  For \(s\in \RR\), define \(U_{s}\colon
  H^{s}_{\ell}(G)\to\ell^{2}(G)\) by
  \begin{equation}
    (U_{s}f)(x) = f(x)(1+\ell(x))^{s/2}.
  \end{equation}
  This is clearly a unitary isomorphism. Now, if \(s>t\), and \(J\colon
  H^{s}_{\ell}(G)\to H^{t}_{\ell}(G)\) is the canonical embedding,
  then
  \begin{equation}
    (U_{t}J U_{s}^{-1}f)(x) = f(x)(1+\ell(x))^{(t-s)/2}.
  \end{equation}
  The compactness of \(J\) is thus equivalent to the compactness of
  the above multiplication operator. It is compact if and only if two
  conditions are satisfied:
  \begin{enumerate}
  \item every value of \((1+\ell(x))^{(t-s)/2}\) is attained only
    finitely many times,
  \item the only accumulation point of \((1+\ell(x))^{(t-s)/2}\) is \(0\).
  \end{enumerate}
  Since \((t-s)/2 < 0\), these two conditions are equivalent to
  properness of \(\ell\).
\end{proof}

Since we are already here, let us prove a corollary of this
observation.

\begin{corollary}
  If \(G\) is infinite, and \(\ell\) is proper, then the space of
  rapidly decreasing functions \(H^{\infty}_{\ell}(G)\) is not a
  Banach space.
\end{corollary}

\begin{proof}
  To the contrary, suppose that \(H^{\infty}_{\ell}(G)\) is a Banach
  space, and consider the identity operator \(I\colon
  H^{\infty}_{\ell}(G)\to H^{\infty}_{\ell}(G)\). By Lemma
  \ref{thm:inv-lim-operator}, it extends to an operator \(I'\colon
  H^{s}_{\ell}(G)\to H^{\infty}_{\ell}(G)\) for some \(s\). But then
  we have the following factorization of \(I\)
  \begin{equation}
    H^{\infty}_{\ell}(G) \xrightarrow{J_{s+1}} H^{s+1}_{\ell}(G)
    \xrightarrow{J_{s+1,s}} H^{s}_{\ell}(G) \xrightarrow{I'} H^{\infty}_{\ell}(G),
  \end{equation}
  from which it follows that \(I\) is compact. This is only possible
  when \(H^{\infty}_{\ell}(G)\) is finite-dimensional, i.e.\ when
  \(G\) is finite.
\end{proof}

\subsection{Decay of matrix coefficients}
\label{sec:decay-matr-coeff}

Let \(\pi\colon G\to {\mathcal{U}}(\mathcal{H})\) be a unitary
representation of \(G\). It induces a seminorm \(\norm{\cdot}_{\pi}\)
on \(\CC[G]\) given by
\begin{equation}
  \norm{f}_{\pi} = \norm{\pi(f)}_{op},
\end{equation}
where \(\pi(f)\) is obtained by linear extension of \(\pi\) to
\(\CC[G]\). We say that \(\pi\) is \emph{tempered} if for all
\(f\in\CC[G]\) we have \(\norm{f}_{\pi}\leq \norm{f}_{op}\).

\begin{definition}\label{def:decay-matrix-coeffs}
  The group \(G\) has property RD with respect to \(\ell\) if there
  exist constants \(C,s>0\) such that for every tempered
  representation \(\pi\colon G\to \mathcal{U}(\mathcal{H})\) of \(G\),
  and for all \(\xi,\eta\in \mathcal{H}\) the following estimate
  holds:
  \begin{equation}\label{eq:coeff-decay-estimate}
    \left( \sum_{x\in G} \frac{\abs{\langle \pi(x) \xi,
          \eta\rangle}^{2}}{(1+\ell(x))^{s}} \right)^{1/2} \leq C \norm{\xi}\norm{\eta}.
  \end{equation}
\end{definition}

\begin{proposition}
  Definition \ref{def:decay-matrix-coeffs} is implied by Definition \ref{def:h-s}, and it implies Definition \ref{def:norm-estimate}.
\end{proposition}

\begin{proof}
  First, assume that \(G\) has property RD with respect to a length
  \(\ell\) in the sense of Definition~\ref{def:decay-matrix-coeffs}, with constants \(C,s>0\). If we take \(f\in\CC[G]\),
  and \(g,h\in\ell^{2}(G)\), we may compute
  \begin{equation}
    \begin{split}
      \abs{\langle f*g, h\rangle} & = \abs{ \sum_{x\in G}\sum_{y\in
          G}f(y)g(y^{-1}x)\overline{h(x)}} = \abs{\sum_{\ell({y})\leq
          \ell(f)} f(y) \langle \lambda(y)g,h\rangle} \leq \\
      & \leq \norm{f}_{2}
      \left(\sum_{{\ell(y)}\leq\ell(f)}\abs{\langle
          \lambda(y)g,h\rangle}^{2}\right)^{1/2} \leq \\
      & \leq \norm{f}_{2}
      \left(\sum_{{\ell(y)}\leq\ell(f)}\abs{\langle
          \lambda(y)g,h\rangle}^{2}\frac{(1+\ell(f))^{s}}{(1+\ell(y))^{s}}\right)^{1/2}
      \leq \\
      & \leq \norm{f}_{2}(1+\ell(f))^{s/2}C\norm{g}_{2}\norm{h}_{2},
    \end{split}
  \end{equation}
  obtaining an estimate for the operator norm of \(f\).

  For the second implication, assume that Definition~\ref{def:h-s} is
  satisfied, and for some \(s\) there is a bounded operator \(J\colon
  H^{s}_{\ell}(G)\to C^{*}_{r}(G)\) extending the inclusion of
  \(\CC[G]\). Take a tempered representation \(\pi\) of \(G\) on a
  Hilbert space \(\mathcal{H}\), and \(\xi,\eta\in \mathcal{H}\), and
  denote the corresponding matrix coefficient by \(\phi\),
  \begin{equation}
    \phi(x)=\langle \pi(x)\xi,\eta\rangle.
  \end{equation}
  As a bounded function, \(\phi\) is a functional on
  \(\ell^{1}(G)\subseteq C_{r}^{*}(G)\), and we claim that it
  extends to a bounded functional \(\tilde{\phi}\)  on \(C^{*}_{r}(G)\) of norm at most
  \(\norm{\xi}\norm{\eta}\). Indeed, take \(f\in \ell^{1}(G)\),
  and observe that
  \begin{equation}
    \begin{split}
      \abs{\sum_{x\in G}f(x)\phi(x)} & = \abs{\sum_{x\in
          G}f(x)\langle\pi(x)\xi,\eta\rangle} = \abs{\langle
        \pi(f)\xi,\eta \rangle}\leq \\
      & \leq \norm{\pi(f)}_{op} \norm{\xi}\norm{\eta} \leq \norm{f}_{op} \norm{\xi}\norm{\eta}.
    \end{split}
  \end{equation}
  Now, consider the adjoint map \(J^{*}\colon (C^{*}_{r}(G))^{*}\to
  (H^{s}_{\ell}(G))^{*}\). It is easy to see that under the pairing
  \((f,g)\to \sum_{x\in G}f(x)g(x)\), the dual space of
  \(H_{\ell}^{s}(G)\) is \(H_{\ell}^{-s}(G)\), and since all the maps
  involved commute with the canonical inclusions of \(\CC[G]\) into
  suitable spaces, as a function, \(J^{*}\tilde{\phi}\) is just
  \(\phi\). We therefore obtain the estimate
  \begin{equation}
    \norm{\phi}_{\ell,-s} \leq \norm{J^{*}}\norm{\tilde{\phi}} \leq \norm{J^{*}}\norm{\xi}{\norm{\eta}},
  \end{equation}
  which is exactly the estimate \eqref{eq:coeff-decay-estimate} in
  Definition \ref{def:decay-matrix-coeffs}.
\end{proof}

\section{Groups with property RD}
\label{sec:groups-with-property}

Now we will say something more about the class of groups with property
RD. We will describe its closure properties, give a geometric
criterion for proving RD, and some further examples of groups with
property RD.

\subsection{Permanence properties}
\label{sec:perm-prop}

The following permanence properties were established by P. Jolissaint
in \cite{Jolissaint1990}

\begin{theorem}\label{thm:permanence}
Suppose that \(G\) and \(H\) are groups with property RD. Then
\begin{enumerate}
\item if \(K\leq G\), then \(K\) has property RD,
\item if \(G \leq K\), and \([K:G]<\infty\), then \(K\) has property
  RD,
\item if \( A \) is finite, then the amalgamated free product
  \(G*_{A}H\) has property RD,
\item if \(A\) is a group with property RD, which embeds as a finite
  index central subgroup of both \(G\) and \(H\), then \(G*_{A}H\) has
  property RD,
\item if \(1\to N\to E\to Q\to 1\) is a short exact sequence of
  groups, such that \(N\) and \(Q\) have property RD, and moreover
  \(N\) is polynomially distorted in \(E\), i.e.\ for \(n\in N\)
  \begin{equation}
    \ell_{N}(n)\leq P(\ell_{E}(n))
  \end{equation}
  for some polynomial \(P\),
  then \(E\) has property RD.
\end{enumerate}
\end{theorem}

Recently, the following strengthening of condition (5) in
Theorem~\ref{thm:permanence}, omitting the assumption of polynomial distortion, was proved in \cite{Garncarek2013a}
(actually, in the context of topological compactly generated groups).

\begin{theorem}\label{thm:extensions}
  Suppose that \(1\to N\to G\to Q\to 1\) is a short exact sequence of
  finitely generated groups, and that
  \begin{enumerate}
  \item \(Q\) has property RD,
  \item \(N\) has property RD with respect to the restriction of \(\ell_{G}\).
  \end{enumerate}
  Then \(G\) has property RD.
\end{theorem}

Note, that without assumptions on the distortion of \(N\), we cannot
use the original length of \(N\), as the following example
shows. Namely, consider the extensions of the form \(1\to\ZZ\to
G\to\ZZ\to 1\). The group \(\ZZ\) has property RD, and \(G\) is
solvable, hence amenable. However, it can have exponential growth, in
which case it does not satisfy property RD. In this case,
Theorem~\ref{thm:extensions} implies that \(\ZZ\) does not satisfy
property RD with respect to the restriction of the length of \(G\).

\subsection{The centroid property}
\label{sec:centroid-property}

In the remainder of this section, we will follow the article
\cite{Sapir2015} by Mark Sapir, which gives a nice common
generalization of previously known geometric approaches to property
RD.

Let \(G\) act by isometries on a metric space \((X,d)\), with
uniformly bounded stabilizers. A \emph{centroid map} is a map
\(m\colon G^{2}\to X\) such that for some polynomial \(P\) the
following conditions are met:
\begin{enumerate}
\item \(\abs{m(B_{G}(r)\times \{h\})} \leq P(r)\) for all \(h\in G\),
\item \(\abs{m(\{g\}\times G)} \leq P(r)\) for all \(g\in B_{G}(r)\),
\item \(\abs{\set{ g^{-1}m(g,gh) : g\in B_{G}(r) }} \leq P(r)\) for
  all \(h\in G\).
\end{enumerate}
We will make an attempt of explaining the meaning of these conditions
later on, before formulating the relative centroid property, and for
now let us just accept them as three necessary ingredients in the
proof of the following result.

\begin{theorem}\label{prop:centroid-rd}
  If \(G\) admits a centroid map, then it has property RD.
\end{theorem}

\begin{proof} 
  Take \(f\in\CC[G]\) with \(\ell(f)=r\) and \(\xi\in \ell^{2}(G)\), and estimate the
  norm of the convolution \(f*\xi\) as follows.
  \begin{equation}
    \begin{split}
      \norm{f*\xi}_{2}^{2} & = \sum_{k\in G} \abs{\sum_{g\in G}
        f(g)\xi(g^{-1}k)}^{2} = \sum_{k\in G} \abs{ \sum_{x\in
          X} \sum_{\substack{\ell(g)\leq r\\ m(g,k)=x}}
        f(g)\xi(g^{-1}k)}^{2} \leq \\
      & \leq \sum_{k\in G} \abs{m(B(r)\times\{k\})}
      \sum_{x\in m(B(r)\times\{k\})} \abs{
        \sum_{\substack{\ell(g)\leq r\\ m(g,k)=x}}
        f(g)\xi(g^{-1}k)}^{2} \leq \\
      & \leq P(r)\sum_{\substack{k\in G\\x\in m(B(r)\times\{k\})}}
      \sum_{\substack{\ell(g)\leq r\\ m(g,k)=x}} \abs{f(g)}^{2} \sum_{\substack{\ell(g)\leq r\\ m(g,k)=x}}
 \abs{\xi(g^{-1}k)}^{2}.
    \end{split}
  \end{equation}
  In the last sum we substitute \(h=g^{-1}k\), i.e.\ \(g=kh^{-1}\),
  obtaining
  \begin{equation}
    \begin{split}
      \norm{f*\xi}_{2}^{2} & \leq P(r)\sum_{\substack{k\in G\\x\in m(B(r)\times\{k\})}}
      \sum_{\substack{\ell(g)\leq r\\ m(g,k)=x}} \abs{f(g)}^{2} \sum_{\substack{\ell(kh^{-1})\leq r\\ m(kh^{-1},k)=x}}
 \abs{\xi(h)}^{2} \leq \\
 & \leq P(r)\sum_{\substack{g,h\in G\\\ell(g)\leq r}} c_{gh} \abs{f(g)}^{2}\abs{\xi(h)}^{2},
    \end{split}
  \end{equation}
  where
  \begin{equation}
    c_{gh} = \abs{ \set{ (k,x) \in
     G\times X : \ell(kh^{-1})\leq r, x=m(g,k)=m(kh^{-1},k) }}.
  \end{equation}
  Now fix \(g\) and \(h\). If the pair \((k,x)\) counts in \(c_{gh}\), then \(x\in
  m(\{g\}\times G)\), and since \(\ell(g) \leq r\), there are at most
  \(P(r)\) possibilities for \(x\). We also have
  \begin{equation}
    hk^{-1}x \in \{ hk^{-1}m(kh^{-1},kh^{-1}h) : \ell(kh^{-1})\leq r\},
  \end{equation}
  so there are \(P(r)\) possibilities for \(hk^{-1}x\). If \(K\) is
  the bound on the stabilizers of the action of \(G\) on \(X\), then
  for every \(x\) we have at most \(KP(r)\) possibilities for
  \(hk^{-1}\), or for \(k\), because \(h\) is fixed. In other words,
  \begin{equation}
    c_{gh} \leq KP(r)^{2},
  \end{equation}
  and
  \begin{equation}
    \norm{f*\xi}_{2}^{2}\leq KP(r)^{3} \norm{f}_{2}^{2} \norm{\xi}_{2}^{2},
  \end{equation}
  yielding property RD for \(G\).
\end{proof}

During the mini-course a question was asked, whether the assumption of
uniformly bounded stabilizers can be weakened. As we see in the above
proof, the only place where it is used is the estimation of the number
of possibilities for \(hk^{-1}\in G\) when knowing that the number of possible
values of \(hk^{-1}x\in X\) for a fixed \(x\in m(\{g\}\times
G)\subseteq m(B(r)\times G)\) is at most
\(P(r)\). This number is bounded by \(\abs{\Stab_{G}(x)}P(r)\), and so in order to
retain a polynomial estimate in \(r\), we need to know that there
exists a polynomial \(Q\) such that
\begin{equation}
  \abs{\Stab_{G}(x)} \leq Q(r)
\end{equation}
for all \(x\in m(B(r)\times G)\). We can call this condition
\emph{polynomially bounded stabilizers (with respect to \(m\))}. If it holds, the rest of the
proof works with \(K\) replaced by \(Q(r)\). Although this condition
is weaker than uniformly bounded stabilizers, it looks more
technical. It is interesting if there exists an example where a
centroid map can be constructed for an action with polynomially
bounded stabilizers, while it is more difficult to get a centroid map
for an action with uniformly bounded stabilizers.

\subsection{Examples of groups with the centroid property}
\label{sec:examples-groups-with}

Here we will give examples of two classes of groups which have
property RD thanks to satisfying the centroid property.

\begin{proposition}\label{prop:hyp-rd}
  Gromov hyperbolic groups have property RD.
\end{proposition}

\begin{proof}
  Suppose that \(G\) is \(\delta\)-hyperbolic, i.e.\ any side of a
  geodesic triangle in the Cayley graph of \(G\) with respect to some
  fixed finite generating set \(S\) is contained in the union of the
  \(\delta\)-neighborhoods of the remaining two sides.

  Let \(X\) be a Cayley graph of \(G\) with the standard
  \(G\)-action. On \(G\) take the length induced from the path metric
  on \(X\). For any \(g\in G\) fix an oriented geodesic
  \([1,g]\subseteq X\) from \(1\) to \(g\), and extend equivariantly
  to a choice of oriented geodesics joining any two \(g,h\in G\).

  Now, for every pair \(g,h\in G\) pick
  \begin{equation}
    m(g,h) \in N_{\delta}([1,g])\cap N_{\delta}([g,h])\cap [1,h].
  \end{equation}
  We claim that \(m\) is a centroid map for \(G\). Indeed, we have
  \begin{align}
    m(B(r)\times\{h\}) &\subseteq \bigcup_{g\in B(r)} N_{\delta}([1,g])
    \cap [1,h] = B_{X}(1,r+\delta)\cap[1,h],\\
    m(\{g\}\times G) &\subseteq N_{\delta}([1,g]),\\
    \begin{split}
      \{g^{-1}m(g,gh) : g\in B(r)\} & \subseteq \bigcup_{g \in B(r)}
      N_{\delta}([g^{-1},1]) \cap N_{\delta}([1,h]) = \\
      & = B_{X}(1,r+\delta)\cap N_{\delta}([1,h]),
    \end{split}
  \end{align}
  and all three sets have cardinalities bounded by a linear polynomial
  in \(r\), \(\abs{g}\), and \(r\), respectively. Hence, \(G\) has
  property RD.
\end{proof}

Before we prove RD for the second class of groups, we will explain
some facts about CAT(0) cube complexes. Roughly speaking, a cube
complex is a metric space obtained by gluing euclidean unit cubes \([0,1]^{n}\)
by isometries of walls. There is a nice combinatorial criterion for
the resulting complex \(K\) to be a CAT(0) metric space, provided that
it is already simply connected. Namely this
is equivalent to the links of vertices of \(K\) being flag simplicial
complexes. In other words, whenever at a vertex \(v\) of \(K\) you see
something that looks like the neighborhood of a vertex in the boundary
of a cube (which corresponds to a boundary of a simplex in the link of
\(v\)), this is because in fact there is a cube in \(K\), whose boundary you
are looking at (and which corresponds to the simplex whose boundary
you see in the link).

From this description it follows that a CAT(0) cube complex is
uniquely determined by its \(1\)-skeleton \(X^{(1)}\), consisting of
vertices and edges of the cubes of \(X\). Indeed, to reconstruct
\(X\), one just has to glue in a cube into every subgraph of \(X^{(1)}\) isomorphic to the
\(1\)-skeleton of a cube.

Not all graphs are \(1\)-skeletons of CAT(0) cube complexes, but there
is again a nice criterion for this. If \(u,v\) are vertices of a graph
\(\Gamma\), denote by \([u,v]\) the set of all vertices of \(\Gamma\)
lying on some path between \(u\) and \(v\) of minimal length. Then
\(\Gamma\) is the \(1\)-skeleton of some CAT(0) cube complex if and
only if for all triples of vertices \(u,v,w\) of \(\Gamma\) the
intersection \([u,v]\cap[v,w]\cap[w,u]\) consists if exactly one
vertex \cite{Chepoi2000}.

An \(n\)-dimensional cube contains \(n\) distinct \emph{midcubes} of the form
\([0,1]^{k}\times\{1/2\}\times[0,1]^{n-k-1}\). If \(K\) is a cube
complex containing two adjacent cubes \(Q_{1}\) and \(Q_{2}\), we say
that two midcubes \(M_{i}\subseteq Q_{i}\) are compatible if
\(M_{1}\cap Q_{2} = M_{2}\cap Q_{1} \ne \emptyset\). In other words,
both midcubes have the same nontrivial intersection with the common
walls of \(Q_{1}\) and \(Q_{2}\). This can be extended to an
equivalence relation on the set of all midcubes in \(K\). The unions
of equivalence classes of this relation are called \emph{hyperplanes}
of \(K\).

If \(K\) is a CAT(0) cube complex, then every hyperplane separates
\(K\) into two components. Moreover, if we denote by \(d\) the path metric on
\(K^{(1)}\), then for any two vertices \(v\) and \(w\) of \(K\), the
distance \(d(v,w)\) is equal to the number of hyperplanes separating
\(v\) and \(w\). Finally, if \(h_{1}\ldots,h_{n}\) is a family of
pairwise intersecting hyperplanes of \(K\), they have non-empty common
intersection \(\bigcap h_{i}\). If this intersection contains a point
in the interior of a cube \(Q\), and \(h_{i}\) are pairwise distinct,
then each \(h_{i}\) contains a distinct midcube of \(Q\), and thus \(n\) is
bounded by \(\dim Q\) \cite{Sageev1995}.

\begin{proposition}\label{prop:catzero-cube-rd}
  If \(G\) acts cellularly on a finite-dimensional CAT(0) cube complex
  with uniformly bounded stabilizers, then \(G\) has property RD.
\end{proposition}

\begin{proof}
  Denote by \(K\) the complex acted upon by \(G\). Let \(X=K^{(0)}\)
  with the distance induced from \(K^{(1)}\). We choose a base vertex
  \(o\in X\), fix the length function \(\ell(g)=d(o,go)\), and define
  \(m(g,h)\) to be the unique vertex in the intersection
  \([o,go]\cap[o,ho]\cap[go,ho]\). We have the following inclusions:
  \begin{align}
    m(B(r)\times\{h\}) & \subseteq B_{X}(o,r) \cap [o,ho], \\
    m(\{g\}\times G) &\subseteq B_{X}(o,\ell(g))\cap [o,go], \\
    \{ g^{-1}m(g,gh) : g\in B(r)\} & \subseteq B_{X}(o,r) \cap [o,ho].
  \end{align}

  We therefore need to estimate the growth \(r\mapsto \abs{[v,w]\cap
    B_{X}(v,r)}\) of intervals in \(K\). We will achieve this by
  introducing a coordinate system in \([v,w]\). Let \(H\) be the set
  of hyperplanes of \(K\) separating \(v\) from \(w\). We may put a
  partial order on \(H\) in the following way: if \(h_{1},h_{2}\in H\),
  we have \(h_{1} < h_{2}\) if and only if \(h_{1}\cap h_{2} =
  \emptyset\), and moreover, \(h_{1}\) separates \(v\) from
  \(h_{2}\). This ordering has the property that if \(p\) is a minimal path
  from \(v\) to \(w\), and \(<_{p}\) is the order on \(H\) in which
  \(p\) crosses the hyperplanes, then \(<_{p}\) extends \(<\). This
  means that if \(C\subseteq H\) is a chain with respect to \(<\), and \(u\in [v,w]\), then
  the set of hyperplanes from \(C\) separating \(v\) from \(u\) is an
  initial segment of \(C\), and therefore it is determined by its cardinality.

  Now, decompose \(H\) into a disjoint union \(H=C_{1}\cup\ldots\cup
  C_{N}\) of chains. For \(u\in [v,w]\) let \(c_{i}(u)\) be the number
  of hyperplanes in \(C_{i}\) separating \(u\) from \(v\). The
  sequence of numbers \((c_{i}(u))_{i\leq N}\) allows to reconstruct
  the set of all hyperplanes separating \(u\) from \(v\), which
  uniquely determines \(u\) (any two distinct vertices are separated
  by a hyperplane, which separates exactly one of them from
  \(v\)). Therefore, the assignment \(u\mapsto (c_{i}(u))_{i\leq N}\in
  \NN^{N}\) is injective. But if \(u\in B_{X}(v,r)\), then
  \(c_{i}(u)\leq r\), so the cardinality of \(B_{X}(v,r)\cap[v,w]\) is
  bounded by \(r^{N}\).

  It now remains to find a bound on \(N\) independent of the choice of
  the interval \([v,w]\). Observe, that antichains in \((H,<)\) have
  at most \(\dim K\) elements. By Dilworth's Theorem, the maximal
  cardinality of an antichain equals to the minimal cardinality of a
  decomposition into chains, so we may take \(N=\dim K\), which ends
  the proof.
\end{proof}

\subsection{The relative centroid property}
\label{sec:relat-centr-prop}

Before defining the relative centroid property, we will try to give
a more symmetric formulation of the centroid property. As before, let
\(G\) act by isometries on \((X,d)\) with uniformly bounded
stabilizers.

Observe that the maps from \(G^{n}\) to \(X\) are in one-to-one
correspondence with \(G\)-equivariant maps from \(G^{n+1}\) to \(X\).
To a map \(f\colon G^{n}\to X\) one just assigns the equivariant map
\begin{equation}
  \tilde{f}(g_{0},g_{1},\ldots,g_{n}) = g_{0}f(g_{0}^{-1}g_{1},\ldots,
  g_{0}^{-1}g_{n}).
\end{equation}
In this way, we may replace the centroid map \(m\colon G^{2}\to X\) by
its equivariant extension \(\tilde{m}\colon G^{3}\to X\). The three
estimates from the definition of the centroid map translate into
\begin{enumerate}\setlength{\itemsep}{3pt}
\item \(\abs{\tilde{m}(\{g_{1}\} \times B(g_{1},r) \times \{g_{3}\})} \leq
  P(r)\) for all \(g_{1},g_{3}\in G\),
\item \(\abs{ \tilde{m}( \{g_{1}\} \times \{ g_{2} \} \times G ) } \leq P(r)\)
  for all \(g_{1}\in G\) and \(g_{2}\in B(g_{1},r)\),
\item \( \abs{ \tilde{m}( B(g_{2},r)\times\{g_{2}\}\times\{g_{3}\} ) } \leq
  P(r) \) for all \(g_{2},g_{3}\in G\).
\end{enumerate}
We can go even further with our symmetrization. Define
\begin{equation}
\Delta(r) = \{(g_{1},g_{2},g_{3})\in G^{3} : \ell(g_{1}^{-1}g_{2})
\leq r\},
\end{equation}
i.e.\ the set of ordered triangles in \(G\), whose first side has
length at most \(r\). The asymmetry can now be hidden in the
definition of \(\Delta(r)\), and we can reformulate conditions (1)--(3)
as
\begin{enumerate}\setlength{\itemsep}{3pt}
\item \( \abs{ \tilde{m}( \Delta(r) \cap \{g_{1}\}\times G \times \{ g_{3}\} )
  } \leq P(r) \),
\item \( \abs{ \tilde{m}( \Delta(r) \cap \{g_{1}\} \times \{ g_{2}\}\times G )
  } \leq P(r) \),
\item \( \abs{ \tilde{m}( \Delta(r) \cap G \times \{g_{2}\}\times \{ g_{3}\} )
  } \leq P(r) \).
\end{enumerate}

Now, let \(T_{1},\ldots, T_{m}\subseteq X^{3}\) be \(G\)-invariant
subsets, and let \(p_{ij}\colon X_{1}\times X_{2}\times X_{3}\to
X_{i}\times X_{j}\) denote the standard projection of the cartesian
product. A \emph{relative centroid map} is a \(G\)-equivariant map
\(\rc \colon G^{3} \to \bigcup_{i} T_{i}\) for which there exists a
polynomial \(P\) such that for all \(g_{1},g_{2},g_{3}\in G\)
\begin{enumerate}\setlength{\itemsep}{3pt}
\item if \(\rc(g_{1},g_{2},g_{3}) = (x_{1},x_{2},x_{3})\), then
  \(d(x_{i},x_{j})\leq P(\ell(g_{i}^{-1}g_{j}))\) for \(i,j=1,2,3\)
\item \(\abs{ p_{12}\circ\rc( \Delta(r) \cap \{g_{1}\} \times \{ g_{2}\} \times G) } \leq P(r)\),
\item \(\abs{ p_{13}\circ\rc( \Delta(r) \cap \{g_{1}\} \times G \times \{ g_{3}\}) } \leq P(r)\),
\item \(\abs{ p_{23}\circ\rc( \Delta(r) \cap G \times \{g_{2}\} \times
    \{ g_{3}\}) } \leq P(r)\).
\end{enumerate}
This time, instead of points, to each ordered triangle in \(G\) we
equivariantly assign an ordered triangle in \(X\), belonging to
a certain prescribed class. The estimates say that if among all
triangles with the first side of length at most \(r\) we fix one side,
and allow the remaining vertex to vary, then in the assigned triangle,
the ``image'' of the fixed side is limited to \(P(r)\) possibilities.
Also, note that if we take \(m=1\) and \(T_{1}=\{(x,x,x) : x\in X\}\),
then the corresponding relative centroid maps reduce to ordinary centroid maps defined in section~\ref{sec:centroid-property}.

We say that \(G\) has \emph{relative centroid property} with respect
to subgroups \(H_{1},\ldots,H_{m}\) if it admits a relative centroid
map with \(X=G^{3}\) and \(T_{i} = G\cdot H_{i}^{3}\).

\begin{theorem}
  Suppose that the length on \(G\) is proper, and that \(G\) has the
  relative centroid property with respect to subgroups \(H_{1},\ldots,
  H_{m}\), such that each \(H_{i}\) has property RD with respect to
  the restriction of the length of \(G\).
\end{theorem}

This criterion can be used to show that a group relatively hyperbolic
with respect to subgroups with property RD, has property RD itself. It
can also be applied to \emph{graph products}. For a finite simple
graph \(\Gamma=(V,E)\) the \(\Gamma\)-product of a family of groups
\(\{G_{v}\}_{v\in V}\) is defined as the quotient of the free product
of the \(G_{v}\) in which elements of groups joined by an edge are
forced to commute:
\begin{equation}
  \prod_{\Gamma} G_{v} = \mathop{\bigstar}\limits_{v} G_{v} / \langle\langle \{ [x,y] :
  x\in G_{v}, y\in G_{w}, \{v,w\}\in E\} \rangle\rangle.
\end{equation}
Any element \(g\) of this graph product can be written in the form
\(g=g_{1}g_{2}\cdots g_{n}\) with \(g_{i}\in G_{v(i)}\) and \(v(i)\ne
v(i+1)\). If \(n\) is minimal, we define
\begin{equation}
  \ell(g) = n+\sum_{i=1}^{n} \ell_{G_{v(i)}}(g_{i}).
\end{equation}

For a clique \(C\subseteq\Gamma\), the \emph{clique subgroup}
\(G_{C} \leq \prod_{\Gamma}G_{v}\) is the direct product \(\prod_{v\in
C} G_{v}\). The graph product \(\prod_{\Gamma}G_{v}\) with the length
\(\ell\) defined above has the relative
centroid property with respect to its clique subgroups, and therefore
if the vertex groups \(G_{v}\) have property RD, so does their graph product.

\section{Some questions}
\label{sec:problems}

\subsection{Property RD for locally compact groups}
\label{sec:property-rd-locally}

Throughout the notes we assumed that \(G\) is a finitely generated group. Property RD can be however defined for locally compact groups. Everything it needs is the Haar measure, allowing to define the the regular representation. In this setting we require length functions to be Borel maps and replace the group algebra \(\CC[G]\) with the space \(C_{c}(G)\) of compactly supported continuous functions on \(G\). One not entirely obvious thing to note here is that any length function is bounded on compact sets \cite{Schweitzer1993}. Therefore, if \(G\) has a compact generating set \(K\), then the associated word-length is proper (preimages of bounded intervals are relatively compact), dominates all other length functions, and plays exactly the same role as word-lengths in the finitely generated case.

It turns out that property RD implies that \(G\) is unimodular \cite{Ronghui1996}. The following was proved in \cite{Jolissaint1990}.

\begin{theorem}
  If \(G\) is a second countable locally compact group with a length function \(\ell\), adsmitting a discrete cocompact subgroup \(\Gamma\) satisfying property RD with respect to \(\ell|_{\Gamma}\), then \(G\) has property RD with respect to \(\ell\).
\end{theorem}

It is an open question whether the converse statement holds.

\begin{conjecture}(Valette's conjecture)
  If \(\Gamma\) is a cocompact lattice in \(G\), and \(G\) has property RD, then so does \(\Gamma\).
\end{conjecture}

\subsection{Degree of rapid decay}
\label{sec:degree-rapid-decay}

For any group \(G\) we may define the \emph{degree of rapid decay} of \(G\) as
\begin{equation}
  \deg_{RD} G = \inf\{ s > 0 : \text{\(\CC[G]\) embeds into \(H_{\ell}^{s}(G)\) for some length \(\ell\)}\}.
\end{equation}
Then \(G\) has property RD if and only if \(\deg_{RD}G\).

In Propositions \ref{thm:poly-growth-rd} and \ref{prop:rd-amenable} we have actually shown that if \(G\) has polynomial growth of degree \(d\), then \(\deg_{RD} G = d\). It is easy to show that \(\deg_{RD}G=0\) if and only if \(G\) is finite. Moreover, in \cite{Nica2010} the following theorem was proved\footnote{in the paper the \(\norm{\cdot}_{\ell,s}\) norm is defined as the \(\ell^{2}\)-norm of \(f(1+\ell)^{s}\), hence our \(\deg_{RD}\) is twice as large as the one used therein.}.

\begin{theorem}
  If \(G\) is an infinite finitely generated group, then \(\deg_{RD}G \geq 1\).
\end{theorem}

It is clear from the definition that if \(H\leq G\), then \(\deg_{RD} H \leq \deg_{RD} G\). Therefore, if \(G\) contains an element of infinite order, then \(\deg_{RD}G \geq \deg_{RD}\ZZ = 1\), and this theorem is non-trivial, provided that there exists a torsion group with property RD. As far as I know, no such group has been discovered.

\begin{question}
  Does there exist an infinite, finitely generated torsion group with property RD?
\end{question}

The proof of Theorem \ref{prop:centroid-rd} gives an upper bound for the degree of rapid decay, namely
\begin{equation}
  \deg_{RD} G \leq 3\deg P,
\end{equation}
where \(P\) is the polynomial from the definition of the centroid map. Actually, the three conditions of the definition could use three different polynomials, \(P_{1}\), \(P_{2}\), and \(P_{3}\), yielding an estimate
\begin{equation}
  \deg_{RD}G \leq \deg P_{1} + \deg P_{2} + \deg P_{3}.
\end{equation}
For instance, if \(G\) is hyperbolic, then by \ref{prop:hyp-rd} one has
\begin{equation}
  \deg_{RD}G \leq 3,
\end{equation}
and if \(G\) acts on an \(n\)-dimensional CAT(0) cube complex with uniformly bounded stabilizers, then
\begin{equation}
  \deg_{RD} G \leq 3n.
\end{equation}
\begin{question}
  What are the possible values of \(\deg_{RD}\)? Can it take non-integer values?
\end{question}

\begin{question}
  Is \(\deg_{RD} G \leq 1\) if and only if \(G\) is virtually cyclic?
\end{question}

\subsection{Property RD for unitary representations}
\label{sec:property-rd-unitary}

Let \(\pi\colon G\to \mathcal{U}(\mathcal{H})\) be a unitary representation of \(G\). We say that \(\pi\) has property RD with respect to a length function \(\ell\) if there exists a polynomial \(P\) such that for all \(f\in\CC[G]\)
\begin{equation}
  \norm{\pi(f)}_{op} \leq P(\ell(f))\norm{f}_{2}.
\end{equation}
This is equivalent to the estimate
\begin{equation}
  \left( \sum_{x\in G} \frac{\abs{\langle \pi(g)\xi,\eta \rangle}^{2}}{(1+\ell(x))^{s}} \right)^{1/2} \leq C\norm{\xi}\norm{\eta}
\end{equation}
holding for some \(C,s>0\), and all \(\xi,\eta\in \mathcal{H}\).

One of the equivalent definitions of weak containment of unitary representations says that \(\sigma\) is weakly contained in \(\pi\) if \(\norm{f}_{\sigma}\leq \norm{f}_{\pi}\) for all \(f\in \CC[G]\). It is then clear that if \(\pi\) has property RD, so do all its weak subrepresentations. As we have seen in Section \ref{sec:decay-matr-coeff}, property RD for \(G\) is equivalent to property RD for its regular representation, or equivalently, all its tempered representations.

\begin{question}
  Is there a group without property RD, which admits a representation satisfying property RD?
\end{question}

A more general question would be:

\begin{question}
  For a given group \(G\), what is the description of the class of its representations satisfying property RD?
\end{question}

One can for instance easily prove that a finite-dimensional representation of \(G\) has property RD if and only if \(G\) has polynomial growth, which gives the first restriction on the class of representations with RD.

\section{Acknowledgements}
\label{sec:acknowledgements}

We wish to than Goulnara Arzhantseva for the invitation to participate in the Measured Group Theory program, thanks to which these notes were created.

\end{document}